\documentclass[12pt,notitlepage]{article}
\usepackage{standalone}
\usepackage{amsmath}
\usepackage{amssymb}
\usepackage{amsbsy}
\usepackage{amsfonts}
\usepackage{titlesec}
\usepackage[all,cmtip]{xy}
\usepackage{amsthm}
\usepackage{mathrsfs}
\usepackage{geometry,graphicx,color}
\usepackage{mathtext}
\usepackage{hyperref}
\usepackage{booktabs}
\theoremstyle{plain}
\newtheorem{theorem}{Theorem}[section]
\newtheorem{lemma}[theorem]{Lemma}

\newtheorem{cor}[theorem]{Corollary}

\newcommand{\xar}[1]{\ensuremath{\xrightarrow{#1}}}
\newcommand{\mf}[1]{\mathfrak{#1}}
\newcommand{\mc}[1]{\mathcal{#1}}

\newcommand{\mb}[1]{\mathbf{#1}}

\newcommand{\Z}{\mathbb{Z}}

\newcommand{\bs}[1]{\mathbf{#1}}

\newcommand{\on}[1]{\operatorname{#1}}
\newcommand{\la}{\langle}
\newcommand{\ra}{\rangle}

\titleformat{\subparagraph}[runin]{\normalsize\bfseries}{\textbf{(\arabic{subparagraph}).}}{1em}{}[]

\newcommand{\p}[1]{\subparagraph{#1}}
\title{Which circle bundles  can be triangulated over \\ $\partial \Delta^3$?}
\author{N. Mn\"ev%
	\thanks{PDMI RAS;  Chebyshev Laboratory, St. Petersburg State University.
		Supported by the Russian Science Foundation grant №14-21-00035. }}
\date{}	
	
\begin{document}
\maketitle
\p{Introduction}
 We can solve the  difficult problem in the title posted long ago by nobody.   It seems  that it leads right into  simplicial wildness of classical combinatorial topology. But it happens  that we can manage it.    The answer is \\ Theorem \ref{theo}\~:\begin{quote}
 	 \textit{Having the boundary of a standard 3-dimensional simplex $\partial \Delta^3$ as the base of  a triangulation,  one can triangulate only  trivial and Hopf circle  bundles.}
 \end{quote} This statement is the immediate output of a really difficult construction by bare hands of a minimal triangulation of Hopf bundle in \cite{Madahar:2000} and the very special probabilistic look of the \textit{rational} local combinatorial formula\footnote{G. Gangopadhyay \cite{Gangopadhyay} invented for the formula  a nice name ``counting triangles formula".}  $\mbox{}^p C_1 (-)$ for the Chern-Euler class of triangulated circle bundles \cite{MS}.
 The proof is the following: oriented circle bundles over sphere are  classified up to isomorphism  by their integer characteristic  Chern-Euler number, and these numbers span entire  $H^2(S^2;\Z) \approx \Z$ (see \cite{Chern1977} \cite[Ch. 2]{Brylinski2008} for a survey). The number measures a sort of winding of the total space over the base. The sign of the number depends only on the orientation whereas the absolute value of the number $\mb c$ is of interest.   This is called ``Gauss-Bonnet-Chern" theorem. The formula  $\mbox{}^p C_1 (-)$ endows a triangulated bundle with a simplicial rational  2-cocycle on the base of the triangulation. The rational cocycle represents the integer Chern-Euler class in ordered simplicial cohomology of the base. The value of the cochain on a 2-simplex is a function of combinatorics of the stalk of the triangulation over that simplex. Chern-Euler integer number of the bundle is the result of the pairing of this rational cochain with the fundamental class.  The key point coming from the formula  $\mbox{}^p C_1 (-)$   is that a stalk of a  triangulation over a simplex gives a contribution $< \frac{1}{2}$ into the rational simplicial  Chern-Euler cochain on the base.   Therefore  having only four 2-simplices in $\partial \Delta^3$ we can, for any triangulation over $\partial \Delta^3$, achieve integer Chern-Euler number $\bs c < 2$. For $\bs c =0$ then we have trivial bundle and we can triangulate it for sure as a product of base and fiber triangulations, when  $\bs c =1$ it is Hopf bundle $S^1 \xar{}S^3 \xar{h} \partial \Delta^3$ and it was triangulated by Madahar and Sarkara. The other Chern-Euler numbers we can't expect.
 It is  interesting to know that Seifert fibrations $S^3 \xar{} S^2$ with any Hopf invariant can be triangulated over $\partial \Delta^3$ \cite{Madahar2002}.  Here we recall a few points from \cite{MS} and add Lemma
 \ref{lem} providing the proof for the central statement.
 \p{Elementary simplicial circle bundles over a $k$-simplex and $(k+1)$-co\-lo\-red necklaces.} \label{word}
 \textit{Elementary simplicial circle bundle} (elementary s.c. bundle) \cite[\S 3]{MS} over an ordered $k$-simplex $\la k \ra$ is
 a map $\mf R \xar{e} \la k\ra $ of a simplicial complex $\mf R$ onto $\la k\ra$, whose geometric realization $|\mf e|$ is a trivial PL fiber bundle over the geometric simplex $\Delta^k$ with fiber $S^1$. We  equip $\mf e$ with orientation - fixed generator of 1-dimensional integer homology of $\mf R$.  Figure \ref{eb} presents a picture of an elementary s.c. bundle over the 1-simplex $\la 1\ra$.
 \begin{figure}[h!]
 	\begin{center}
 		\includegraphics[width=3.0in]{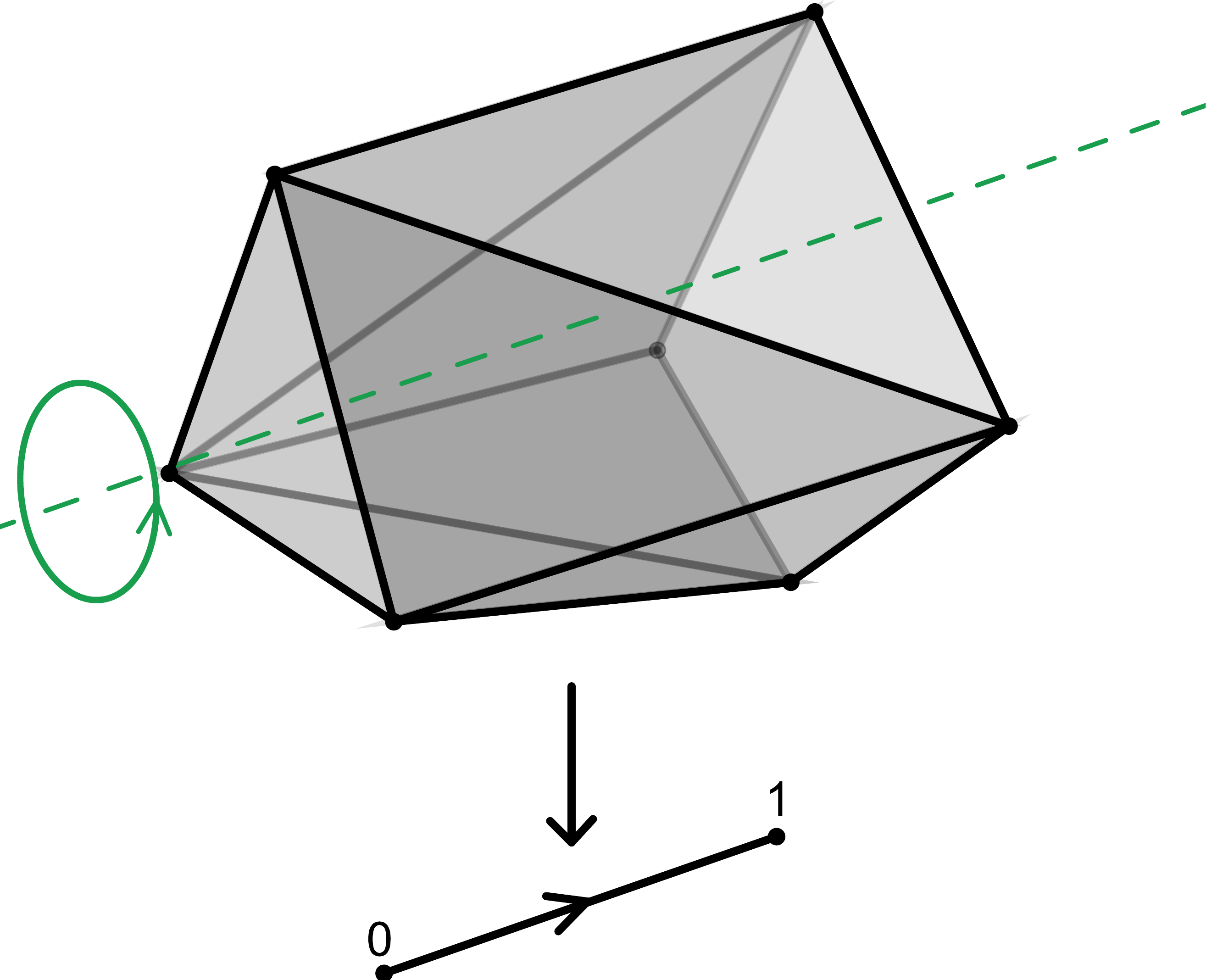} 
 		\caption{ Elementary simplicial circle bundle over the interval. \label{eb}}
 	\end{center}
 \end{figure}
 We are especially interested in elementary s.c. bundles over the $2$-simplex $\la 2 \ra$.
 To an elementary	 simplicial circle bundle $\mf e$ over $\la k \ra$ having $n$ maximal $k+1$-dimensional simplices in the total space,   we associate a $k+1$-necklace $\mc N (\mf e)$, i.e. a length $n$ cyclic word  in the ordered alphabet of $k+1$ letters numbered by the vertices of the base simplex. Any $k+1$ - dimensional simplex of $\mf R$ has a single  edge which shrinks to a vertex $i$ of the base simplex by elementary simplicial degeneration $\la k+1 \ra \xar{\la \sigma_i\ra} \la k\ra, i=0,1,2,...,k$.  Take the general fiber of the projection $|\mf e|$. It is a circle broken into $n$ intervals  oriented by the orientation of the bundle, and every interval on it is an intersection with a maximal $(k+1)$-simplex. The maximal simplex is uniquely named by a vertex of the base where its collapsing edge collapses.  This creates a coloring of the $n$ intervals by $k+1$ ordered vertices of base simplex.  Thus we got a necklace $\mc N (\mf e)$ out of the combinatorics of $\mf e$ (\cite[\S 16]{MS}).  The process is illustrated on Fig \ref{eb2}.

 \begin{figure}[h!]
 	\begin{center}
 		\includegraphics[width=5.0in]{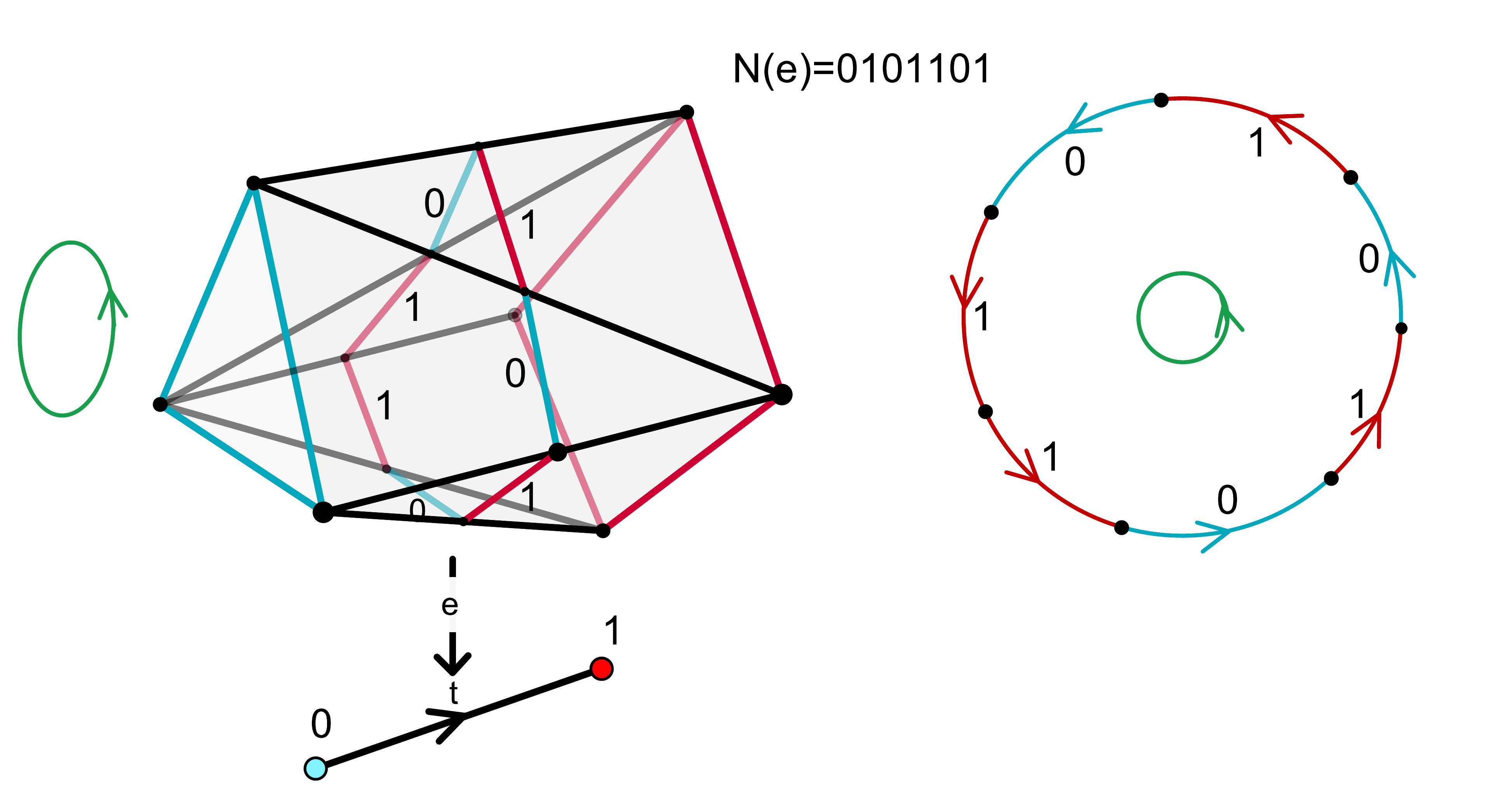}
 		\caption{ \label{eb2}}
 	\end{center}
 \end{figure}

 \p{Parity local formula for Chern-Euler class.}
 Now we switch to the case of 2-di\-men\-sional simplices in the base.
 We have a theorem \cite[Theorem 4.1 ]{MS}.
 A length $n$ word in the alphabet  $[2] = \{0,1,2\}$ in 3 ordered letters can be viewed
 as a surjective map $[n]\xar{w}[2]$ (the map in not required to be monotone).
 Proper subwords of $w$ are sections of this map -- 3-letters subwords, having all the letters different. The proper subword is a permutation of 3 elements and it has a parity -- even or odd.
  We define the rational parity of $w$ as the expectation value of parities of all its proper subwords. Namely, we set
  \begin{equation}\label{par}
  P(w) = \frac{\#(\text{even proper subwords})-\#(\text{odd proper subwords})}{\#(\text{all proper subwords})}
  \end{equation}
 Parity of a permutation of 3 elements is invariant under cyclic shifts.
 Therefore $P(w)$ is invariant under cyclic shifts of the word, and therefore is an invariant of an oriented necklace with beads colored in tree colors ``$0$",``$1$",``$2$".   We define a function of an elementary s.c. bundle over the $2$-simplex
 by the formula
 \begin{equation}\label{chern}
 \mbox{}^p C_1 (\mf e) = -\frac{1}{2} P(\mc N(\mf e))
\end{equation}
The theorem  \cite[Theorem 4.1 ]{MS} states that the rational number  $\mbox{}^p C_1 (\mf e)$ is a local formula for the Chern-Euler class of a triangulated circle bundle.

It follows, by Chern-Gauss-Bonnet theorem, that for any triangulated
circle bundle $\mf E \xar{\mf p } \partial \Delta^3$ over $\partial \Delta^3$, the integer Chern-Euler number
of $|\mf p|$ can be  computed as the value of the parity cocycle (\ref{chern}) on fundamental class. This means the following. Let the simplicial complex $\partial \Delta^3$ be 	ordered and let us fix an orientation of  $\partial \Delta^3$, i.e. the fundamental class is fixed.
Then the absolute Chern-Euler number of $|\mf p|$ is computed as a sum
\begin{equation}\label{sum}
\mb c (|\mf p|) = \on{abs} \left( \sum_{\sigma \in \partial \Delta^3(2)}(-1)^{\on{or}\sigma} (\mbox{}^p C_1(\mf p_\sigma))\right)
\end{equation}
The sum is
running over all four $2$-dimensional ordered simplices of $\sigma \in \partial \Delta^3(2)$. It sums the values of formula $\mbox{}^p C_1 (\mf p_\sigma)$ on elementary subbundles of $\mf p$ over simplices $\sigma$ with  signs
$$(-1)^{\on{or}\sigma}=\begin{cases}1 & \text{if orientation of $\sigma$ consides with global orientaton of $\partial \Delta^3$}  \\
-1 & \text{otherwise}  \end{cases}$$
\p{Extremal parity expectation value.}
In addition to the existing knowledge, we need the following lemma.
 \begin{lemma}\label{lem}
For any elemantary  s.c. bundle $\mf e$ over the 2-simplex,  one has $$\on{abs}(\mbox{}^pC_1 (\mf e)) < \frac{1}{2}$$  	
\end{lemma} 	
\begin{proof} Obviously, $\on{abs}(\mbox{}^pC_1 (\mf e))\leq\frac{1}{2}$.
If $\on{abs}(\mbox{}^pC_1 (\mf e))=\frac{1}{2}$, then  we look at  expectation value  (\ref{par}) and conclude that in the necklace $\mc N (\mf e)$ with probability one all the proper  subwords are simultaneously  even or simultaneously odd.  Therefore the necklace $\mc N(\mf e)$ looks like the word  $$\mc N (\mf e)=0000...1111...2222$$ up to cyclic shifts and reordering of letters. Namely, all the 3 different letters ``0",``1",``2" stays in 3 solid  blocks.   Now let us see what it would mean for the bundle $$\mf R \xar{\mf e} \la 2\ra$$ to have such a necklace.
Consider the face of the base simplex $\la 1\ra \xar{\la \delta_2 \ra} \la 2\ra$ and the induced subbundle $$\delta_2^* \mf e  = \delta_2^* \mf R \xar{\delta_2^* \mf e}\la 1 \ra$$ over the interval $\la 1 \ra$.
In this case (\cite[\S\S 16,17 ]{MS}) the necklace $ \mc N (\delta_2^* \mf e)$ of subbundle $\delta_2^*\mf e$  is the result of deletion of letter ``2" from $\mc N (\mf e)$ and looks as follows: $$\delta_2^* \mc N (\mf e) = 0000011111 \sim 1100000111$$
We may hope to  draw the total complex $\delta_2^*\mf R$ by reverse engineering the process of \S\ref{word}. It is a 2-dimensional picture (see Fig. \ref{le}, where upper and lower sides of the rectangle are identified).
\begin{figure}[h!]
	\begin{center}
		\includegraphics[width=5.0in]{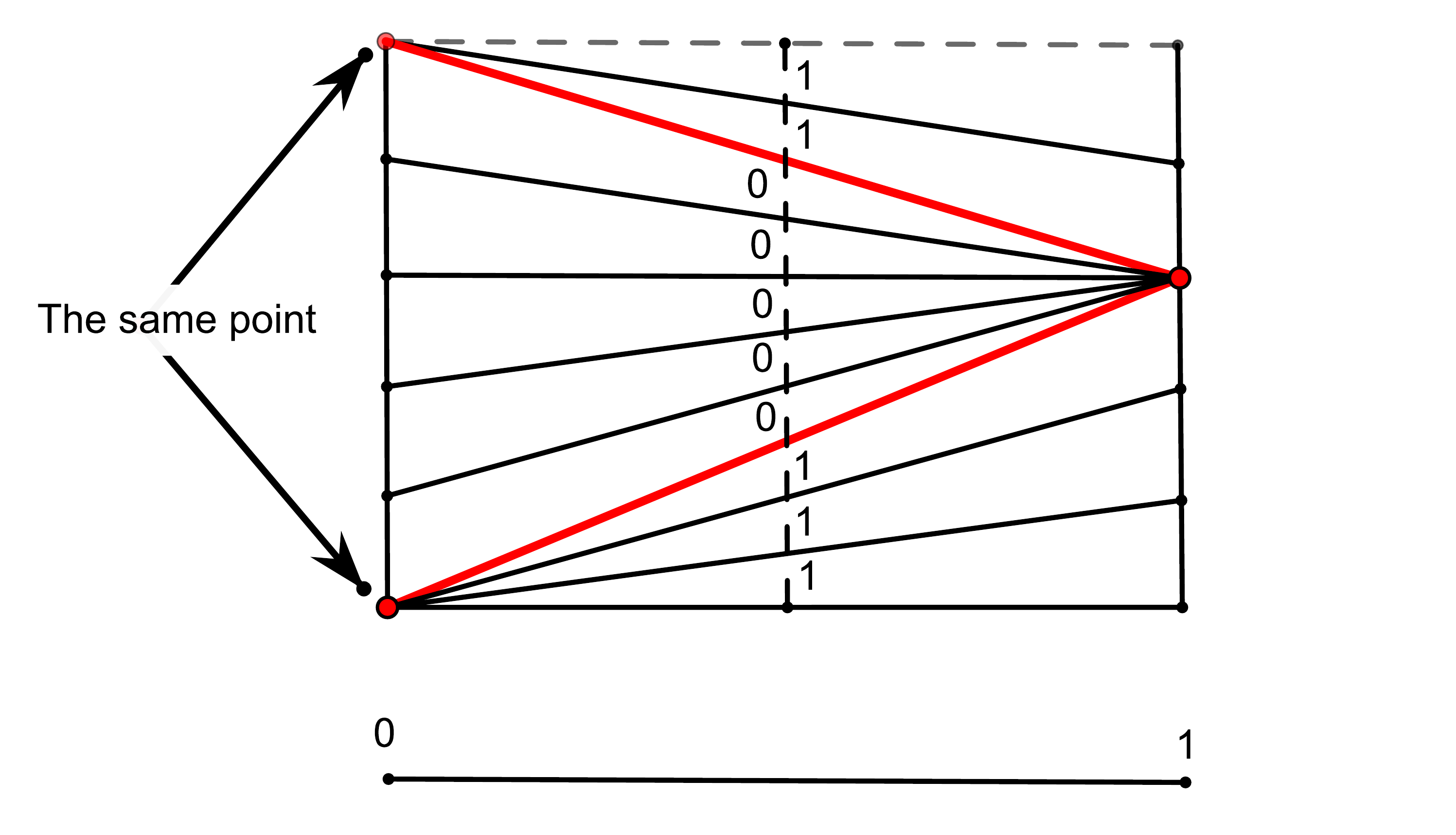}
		\caption{ \label{le}}
	\end{center}
\end{figure}
The edges of   $\delta_2^* \mf R $ which are  projecting surjectively onto the base interval corresponds to pairs of adjacent triangles, which in turn corresponds to pairs of cyclically adjacent letters in the cyclic word. In particular there are two distinguished  edges corresponding to two (cyclic) boundaries of two blocks:
$1(10)000(01)11$.
What we see is that these boundary  edges have both vertices in common. Therefore  $\delta_2^* \mf R $ is not a simplicial complex, therefore $\mf R $  is not a simplicial complex.

We deduced that for an elementary s.c. bundle the value $|\mbox{}^pC_1 (\mf e)|$ cannot be equal to $\frac{1}{2}$ and hence should  be  strictly smaller.
\end{proof}	
\begin{cor} \label{cor2}
A bundle triangulated over $\partial \Delta^3$ should have Chern-Euler number $< 2$	
\end{cor}
\begin{proof}
This follows from Lemma \ref{lem} and our in-house Gauss-Bonnet-Chern formula  (\ref{sum}). The sum of four summands, each having absolute value smaller than $\frac{1}{2}$, has absolute value smaller then 2.  	
\end{proof}
\begin{theorem}\label{theo}
Having the boundary of the standard 3-dimensional simplex $\partial \Delta^3$ as a base of triangulation,  one can triangulate only  trivial and Hopf circle  bundles.
\end{theorem}	
\begin{proof}
From  Corollary \ref{cor2} it follows that we may hope to triangulate only trivial and Hopf bundles. 	
Trivial bundle can be triangulated as the product triangulation of base and fiber; Hopf bundle was triangulated by Madahar and Shakara \cite{Madahar:2000}. 	
\end{proof}	
\p{Semi-simplicial boundary.}
The bundle triangulation on Fig. \ref{le} is ``semi-simplicial." Only semi-simplicial  triangulations realizes the extremal case $\mb c = 2$ -- tangent bundle of $S^2$. We hope to cover these very interesting triangulations somewhere else in more details.  We consider this behavior of the bundle combinatorics in relation to base combinatorics as a nice output of the local formula  (\ref{chern}).

\def\cprime{$'$} \def\cprime{$'$}

\end{document}